\documentclass[12pt]{amsart}

\usepackage{bm,amsfonts,amsmath,amssymb,dsfont,mathrsfs,bbm,fancybox,csquotes} 
\usepackage{times}
\usepackage{graphicx} 
\usepackage{enumerate} 

\newtheorem{theorem}{Theorem}[section]
\newtheorem{definition}[theorem]{Definition}
\newtheorem{prop}[theorem]{Proposition}
\newtheorem{lem}[theorem]{Lemma}
\newtheorem{cor}[theorem]{Corollary}

\theoremstyle{remark}
\newtheorem{remark}[theorem]{Remark}
\makeatletter

\@addtoreset{equation}{section}
\makeatother

\newenvironment{proofof}[1]{\begin{trivlist}
      \item[]\hspace{0cm}{\it Proof of #1.}
      \hspace{0cm}}{\hfill $\square$
      \end{trivlist}}

\newenvironment{rem}{\begin{remark}\rm}{\end{remark}}

\usepackage{color}

\definecolor{gr}{rgb}   {0.,   0.75,   0.23 }
\definecolor{bl}{rgb}   {0.,   0.5,   1. }
\definecolor{cy}{rgb}   {0.,   0.57,   0.67 }
\definecolor{mg}{rgb}   {0.85,  0.,    0.85}
\definecolor{marron}{rgb}  {0.6, 0.40, 0.1} 
\definecolor{or}{rgb}   {0.8,  0.4,   0.}

\definecolor{webred}{rgb}{0.75,0,0}
\definecolor{webgreen}{rgb}{0,0.75,0}
\usepackage[citecolor=webgreen,colorlinks=true,linkcolor=webred]{hyperref}

\newcommand{\N}{\mathbb{N}}
\newcommand{\R}{\mathbb{R}}

\newcommand{\cA}{\mathcal{A}}

\newcommand{\cC}{\mathcal{C}}

\newcommand{\cL}{\mathcal{L}}

\newcommand{\cO}{\mathcal{O}}

\newcommand{\cS}{\mathcal{S}}

\newcommand{\cU}{\mathcal{U}}
\newcommand{\cV}{\mathcal{V}}
\newcommand{\cW}{\mathcal{W}}

\newcommand{\gD}{\mathfrak{D}}

\newcommand{\gP}{\mathfrak{P}}

\newcommand{\gV}{\mathfrak{V}}

\newcommand{\dS}{\mathbb{S}}

\newcommand{\Rob}{\sf Rob}

\newcommand{\bA}{{\boldsymbol{\mathsf{A}}}}
\newcommand{\bB}{{\boldsymbol{\mathsf{B}}}}
\newcommand{\sfA}{{\mathsf{A}}}
\newcommand{\sfB}{{\mathsf{B}}}

\newcommand{\be}{{\boldsymbol{\mathsf{e}}}}
\newcommand{\bff}{{\boldsymbol{\mathsf{f}}}}

\newcommand{\bp}{{\boldsymbol{\mathsf{p}}}}

\newcommand{\bv}{{\boldsymbol{\mathsf{v}}}}

\newcommand{\bx}{{\boldsymbol{\mathsf{x}}}}

\newcommand{\bX}{{\boldsymbol{\mathsf{X}}}}

\newcommand{\bfz}{{\boldsymbol{0}}}

\newcommand{\sC}{\mathscr{C}}
\newcommand{\sE}{\mathscr{E}}

\newcommand{\sL}{\mathrm{L}}
\newcommand{\sH}{\mathrm{H}}
\newcommand{\sB}{\mathrm{B}}

\newcommand{\rd}{\,\mathrm{d}}

\newcommand{\rx}{\mathsf{x}}
\newcommand{\rD}{\mathrm{D}}

\newcommand{\rJ}{\mathrm{J}}
\newcommand{\rL}{\mathrm{L}}
\newcommand{\rX}{\mathsf{X}}

\newcommand{\rN}{\mathrm{N}}
\newcommand{\rP}{\mathrm{P}}
\newcommand{\rR}{\mathrm{R}}

\newcommand{\ess}{\mathrm{ess}}

\newcommand{\ee}{\hskip 0.15ex}

\newcommand\dom{\operatorname{Dom}}

\newcommand\curl{\operatorname{curl}}

\newcommand\Id{\operatorname{\mathbb{I}}}

\newcommand{\OP}{H} % operateur
\newcommand{\En}{E} % ground energy
\newcommand{\seE}{\mathscr{E}^*}

\newcommand{\diffeo}{\mathrm U}

\def\a{\varepsilon}
\def\cone{\cC_{\omega}}
\def\conea{\cC_{\omega_{\a}}}
\def\w{\mathrm{w}}

\def\harm{\mathfrak p}

\title{Magnetic Laplacian in sharp three-dimensional cones}
\author[V. Bonnaillie-No\"el]{Virginie Bonnaillie-No\"el}
\address{Virginie Bonnaillie-No\"el, D\'epartement de Math\'ematiques et Applications (DMA UMR 8553), PSL Research University, CNRS, ENS Paris, 45 rue d'Ulm, F-75230 Paris Cedex 05, France} 
\email{bonnaillie@math.cnrs.fr}
\author[M. Dauge]{Monique Dauge}
\address{Monique Dauge, IRMAR UMR 6625 - CNRS, Universit\'e de Rennes 1, Campus de Beaulieu, 35042 Rennes Cedex, France} 
\email{monique.dauge@univ-rennes1.fr}
\author[N. Popoff]{Nicolas Popoff}
\address{Nicolas Popoff, IMB UMR 5251 - CNRS, Universit\'e de Bordeaux, 351 cours de la lib\'eration, 33405 Talence Cedex, France}
\email{nicolas.popoff@u-bordeaux.fr}
\author[N. Raymond]{Nicolas Raymond}
\address{Nicolas Raymond, IRMAR UMR 6625 - CNRS, Universit\'e de Rennes 1, Campus de Beaulieu, 35042 Rennes Cedex, France} 
\email{nicolas.raymond@univ-rennes1.fr}

\begin{document}

\begin{abstract}
The core result of this paper is an upper bound for the ground state energy of the magnetic Laplacian with constant magnetic field on cones that are contained in a half-space. This bound involves a weighted norm of the magnetic field related to moments on a plane section of the cone. When the cone is sharp, i.e. when its section is small, this upper bound tends to $0$. A lower bound on the essential spectrum is proved for families of sharp cones, implying that if the section is small enough the ground state energy is an eigenvalue. This circumstance produces corner concentration in the semi-classical limit for the magnetic Schr\"odinger operator when such sharp cones are involved. 
\end{abstract}

\today
\maketitle

\section{Introduction}
\subsection{Motivation}
The onset of supraconductivity in presence of an intense magnetic field in a body occupying a domain $\Omega$ is related to the lowest eigenvalues of \enquote{semiclassical} magnetic Laplacians in $\Omega$ with natural boundary condition (see for instance \cite{LuPan00, FouHe06-2, FouHe10}), and its localization is connected with the localization of the corresponding eigenfunctions.

The semiclassical expansion of the first eigenvalues of Neumann magnetic Laplacians has been addressed in numerous papers, considering constant or variable magnetic field. In order to introduce our present study, it is sufficient to discuss the case of a {\em constant magnetic field} $\bB$ and of a simply connected domain $\Omega$. 

For any chosen $h>0$, let us denote by $\lambda_{h}(\bB,\Omega)$ the first eigenvalue of the magnetic Laplacian $(-ih\nabla+\bA)^2$ with Neumann boundary conditions. Here $\bA$ is any associated potential (i.e., such that $\curl\bA=\bB$). The following facts are proved in dimension $2$.
\begin{enumerate}[\em i)]
\item The eigenmodes associated with $\lambda_{h}(\bB,\Omega)$ localize near the boundary as $h\to0$, see \cite{HeMo01}.
\item For a smooth boundary, these eigenmodes concentrate near the points of maximal curvature, see \cite{FouHe06}.
\item In presence of corners for a polygonal domain, these eigenmodes localize near acute corners (i.e. of opening $\le\frac\pi2$), see \cite{Bo05, BoDa06}.
\end{enumerate}
Results {\em i)} and {\em iii)} rely on the investigation of the collection of the ground state energies $E(\bB,\Pi_{\bx})$ of the associated {\em tangent problems}, i.e., the magnetic Laplacians for $h=1$ with the same magnetic field $\bB$, posed on the (dilation invariant) tangent domains $\Pi_{\bx}$ at each point $\bx$ of the closure of $\Omega$. The tangent domain $\Pi_{\bx}$ is the full space $\R^2$ if $\bx$ is an interior point, the half-space $\R^2_+$ if $\bx$ belongs to a smooth part of the boundary $\partial\Omega$, and a sector $\cS$ if $\bx$ is a corner of a polygonal domain. The reason for {\em i)} is the inequality $E(\bB,\R^2_+)<E(\bB,\R^2)$ and the reason for {\em iii)} is that the ground state energy associated with an acute sector $\cS$ is less than that of the half-plane $\R^2_+$. Beyond this result, there also holds the small angle asymptotics (see \cite[Theorem 1.1]{Bo05}), with $\cS_{\alpha}$ the sector of opening angle $\alpha$, 
\begin{equation}
\label{eq:sector}
   E(\bB,\cS_{\alpha}) = \|\bB\| \frac{\alpha}{\sqrt{3}} + \cO(\alpha^3).
\end{equation}
Asymptotic formulas for the first eigenvalue $\lambda_{h}(\bB,\Omega)$ are established in various configurations (mainly in situations {\em ii)} and {\em iii)}) and the first term is always given by
\begin{equation}
\label{eq:asy2}
   \lim_{h\to0} \frac{\lambda_{h}(\bB,\Omega)}{h} = \inf_{\bx\in\overline\Omega} E(\bB,\Pi_{\bx})\,.
\end{equation}

As far as three-dimensional domains are concerned, in the recent contribution \cite{BoDaPo14} formula \eqref{eq:asy2} is proved to be still valid in a general class of corner domains for which tangent domains at the boundary are either half-planes, infinite wedges or genuine infinite 3D cones with polygonal sections. Various convergence rates are proved. Thus the analysis of the Schr\"odinger operator with constant magnetic field on general cones is crucial to exhibit the main term of the expansion of the ground energy of the magnetic Laplacian in any corner domain. 
As in 2D, the interior case $\Pi_{\bx}=\R^3$ ($\bx\in\Omega$) is explicit, and the half-space is rather well known (see \cite{Pan02,HeMo02}). The case of wedges has been more recently addressed in \cite{PoTh,Pof13T,Pop13}.

When the infimum is reached at a corner, a better upper bound of $\lambda_{h}(\bB,\Omega)$ can be proved as soon as the bottom of the spectrum of the corresponding tangent operator is discrete \cite[Theorem 9.1]{BoDaPo14}. If, moreover, this infimum is attained {\em at corners only}, the corner concentration holds for associated eigenvectors \cite[Section 12.1]{BoDaPo14}. 
So the main motivation of the present paper is to investigate 3D cones in order to find sufficient conditions ensuring positive answers to the following questions:
\begin{enumerate}[(Q1)]
\item A 3D cone $\Pi$ being given, does the energy $\En(\bB,\Pi)$ correspond to a discrete eigenvalue for the associated magnetic Laplacian?
\item A corner domain $\Omega\subset \R^{3}$ being given, is the infimum in \eqref{eq:asy2} reached at a corner, or at corners only?
\end{enumerate}

In \cite{Pan02}, positive answers are given to these questions when $\Omega$ is a cuboid (so that the 3D tangent cones are octants), under some geometrical hypotheses on the orientation of the magnetic field. In \cite{BoRay13,BoRay14}, the case of {\em right circular cones} (that we denote here by $\cC^\circ_\alpha$ with $\alpha$ its opening) is investigated: a full asymptotics is proved, starting as
\begin{equation}
\label{eq:conecirc}
   E(\bB,\cC^\circ_\alpha) = \|\bB\| {\sqrt{ 1+\sin^2\beta}} \frac{3\alpha}{4\sqrt{2}} + \cO(\alpha^3),
\end{equation} 
where $\beta$ is the angle between the magnetic field $\bB$ and the axis of the cone. When combined with a positive $\alpha$-independent lower bound of the essential spectrum, such an asymptotics  guarantees that for $\alpha$ small enough, $E(\bB,\cC^\circ_\alpha)$ is an eigenvalue, providing positive answer to Question (Q1).

The aim of this paper is to deal with more general cones, especially with {\em polygonal section}. We are going to prove an upper bound that has similar characteristics as the asymptotical term in \eqref{eq:conecirc}. We will also prove that there exist eigenvalues below the essential spectrum as soon as the cone is {\it sharp} enough, and therefore provide sufficient conditions for a positive answer to Question (Q1).

One of the main new difficulties is that the essential spectrum strongly depends on the dihedral angles of the cones, and that, if these angles get small, the essential spectrum may go to $0$ by virtue of the upper bound
\begin{equation}
\label{eq:wedge}
   E(\bB,\cW_\alpha) \le \|\bB\| \frac{\alpha}{\sqrt{3}} + \cO(\alpha^3),
\end{equation}
where $\alpha$ is the opening of the wedge $\cW_\alpha$. Here the magnetic field $\bB$ is assumed either  to be contained in the bisector plane of the wedge (see \cite[Proposition 7.6]{PoTh}), or to be tangent to a face of the wedge (see \cite[Section 5]{Pof13T}). The outcome of the present study is that eigenvalues will appear under the essential spectrum for sharp cones that do not have sharp edges. 

Obviously, \eqref{eq:wedge} may also be an obstruction to a positive answer to Question (Q2). Combining  our upper bound for sharp cones with the positivity and the continuity of the ground energy on wedges, we will deduce that a domain that has a sharp corner gives a positive answer to (Q2), provided the opening of its edges remained bounded from below. We will also exhibit such a domain by an explicit construction. 

Finally, we can mention that that there exist in the literature various works dealing with spectral problems involving conical domains: Let us quote among others the ``$\delta$-interaction'' Schr\"odinger operator, see \cite{BeExLo14}, and the Robin Laplacian, see \cite{LevPar08}. We find out that the latter problem shares many common features with the magnetic Laplacian, and will describe some of these analogies in the last section of our paper.

\subsection{Main results}
Let us provide now the framework and the main results of our paper.
We will consider cones defined through a plane section.
\begin{definition}
Let  $\omega$ be a bounded and connected open subset of $\R^2$. We define the cone $\cone$ by
\begin{equation}\label{def.cone}
\cone  = \left\{ \bx=(\rx_{1},\rx_{2},\rx_{3})\in \R^3 : \quad\rx_{3}>0\quad\mbox{and}\quad
\left(\frac{\rx_{1}}{\rx_{3}},\frac{\rx_{2}}{\rx_{3}}\right)\in \omega\right\}.
\end{equation}
\end{definition}
 
Let $\bB=(\sfB_{1},\sfB_{2},\sfB_{3})^{\sf T}$ be a constant magnetic field and $\bA$ be an associated linear magnetic potential, i.e., such that $\curl \bA=\bB$.
We consider the quadratic form
\begin{align*}
q[\bA,\cone ](u) &= \int_{\cone}|(-i\nabla+\bA)u|^2 \rd\bx,
\end{align*}
defined on the form domain 
$\dom(q[\bA,\cone ])=\left\{u\in\sL^2(\cone):\ (-i\nabla+\bA)u\in\sL^2(\cone)\right\}$.
We denote by $\OP(\bA,\cone)$ the Friedrichs extension of this quadratic form. If the domain $\omega$ is regular enough (for example if $\omega$ is a bounded polygonal domain), $\OP(\bA,\cone)$ coincides with the Neumann realization of the magnetic Laplacian on $\cone$ with the magnetic field $\bB$. 
By gauge invariance 
the spectrum of $\OP(\bA,\cone)$ depends only on the magnetic field $\bB$ and not on the magnetic potential $\bA$ that is {\it a priori} assumed to be linear.
For $n\in \N$, we define $\En_{n}(\bB,\cone)$ as the $n$-th Rayleigh quotient of $\OP(\bA,\cone)$:
\begin{align}
\label{Rayleigh}
   \En_{n}(\bB,\cone) 
   &=\sup_{u_{1},\ldots,u_{n-1}\in \dom(q[\bA,\cone])} \ \inf_{\substack{u\in[u_{1},\ldots,u_{n-1}]^{\bot}\\ u\in \dom(q[\bA,\cone]) }} 
   \frac{q[\bA,\cone ](u)}{\|u\|^2_{\sL^2(\cone)}}. %\\
\end{align}
For $n=1$, we shorten the notation by $\En(\bB,\cone)$ that is the ground state energy of the magnetic Laplacian $\OP(\bA,\cone)$. 

\subsubsection{Upper bound for the first Rayleigh quotients}
Our first result states an upper bound for $\En_{n}(\bB,\cone)$ valid for any section $\omega$.

\begin{theorem}\label{prop.majo1}
Let $\omega$ be an open bounded subset of $\R^2$ and $\bB$ be a constant magnetic field. We define, for $k=0,1,2$, the normalized moments (here $|\omega|$ denotes the measure of $\omega$)
$$m_{k}:=\frac{1}{|\omega|} \int_{\omega}\rx_{1}^{k}\rx_{2}^{2-k}\rd \rx_{1} \rd \rx_{2}.$$
The $n$-th Rayleigh quotient satisfies the upper bound
\begin{equation}\label{eq:Ee}
\En_{n}(\bB,\cone)\leq (4n-1)e(\bB,\omega),
\end{equation}
where $e(\bB,\omega)$ is the positive constant defined by
\begin{equation}\label{eq:mi}
e(\bB,\omega)
=\left(\sfB_{3}^2\frac{m_{0}m_{2}-m_{1}^2}{m_{0}+m_{2}}
   + \sfB_{2}^2m_{2} + \sfB_{1}^2m_{0} -2 \sfB_{1}\sfB_{2}m_{1} \right)^{1/2}.
\end{equation}
\end{theorem}

\begin{lem}
\label{lem.homog}
There holds
\begin{enumerate}[i)]
\item The application $\bB\mapsto e(\bB,\omega)$ is an $\omega$-dependent norm on $\R^{3}$. 
\item The application $(\bB,\omega)\mapsto e(\bB,\omega)$ is homogeneous:
\begin{equation}
\label{eq:homo}
   e(\bB,\omega) = |\omega|^{1/2}\,\|\bB\|\ e({\bf b},{\varpi}),
   \qquad \mbox{ with}\qquad 
   \mathbf{b}=\frac{\bB}{\|\bB\|},\quad {\varpi}=\frac{\omega}{|\omega|}.
\end{equation}
\end{enumerate}
\end{lem}

\begin{remark}
\label{rem:invar}
a) Although the quantity $e(\bB,\omega)$ is independent of the choice of the Cartesian coordinates $(\rx_1,\rx_2)$ in the plane $\rx_3=0$, it strongly depends on the choice of the $\rx_3$ ``axis'' defining this plane. Indeed, if a cone $\cC$ contained in a half-space is given, there are many different choices possible for coordinates $(\rx_1,\rx_2,\rx_3)$ so that $\cC$ can be represented as \eqref{def.cone}. To each choice of the $\rx_3$ axis corresponds a distinct definition of $\omega$. For instance, let  $\cC$ be a circular cone. If the $\rx_3$ axis is chosen as the axis of the cone, then $\omega$ is a disc. Any different choice of the axis $\rx_3$ yields an ellipse for $\omega$ and the corresponding quantity $e(\bB,\omega)$ would be larger. 

b) When $\omega$ is the disc of center $(0,0)$ and radius $\tan\frac\alpha2$, the cone $\cone$ equals the circular cone $\cC^\circ_\alpha$ of opening $\alpha$ considered in \cite{BoRay13,BoRay14}. Then we find that $e(\bB,\omega)$ coincides with the first term of the asymptotics \eqref{eq:conecirc} modulo $\cO(\alpha^3)$, which proves that our upper bound is sharp in this case (see Section~\ref{subsec.circular} below).
\end{remark}

\subsubsection{Convergence of the bottom of essential spectrum}

By the min-max principle, the quantity $\En_{n}(\bB,\cone)$,  defined in \eqref{Rayleigh}, is either the $n$-th eigenvalue of $\OP(\bA,\cone)$, or the bottom of the essential spectrum denoted by $\En_{\ess}(\bB,\cone)$.

The second step of our investigation is then to determine the bottom of the essential spectrum. We assume that $\omega$ is a bounded polygonal domain in $\R^2$. This means that the boundary of $\omega$ is a finite union of smooth arcs (the sides) and that the tangents to two neighboring sides at their common end (a vertex) are not colinear. Then the set $\cone \cap \dS^{2}$ called the section of the cone $\cone$ is a polygonal domain of the sphere that has the same properties.
For any $\bp \in \overline{\cone} \cap \dS^{2}$, we denote by $\Pi_{\bp}\subset \R^{3}$ the tangent cone to $\cone $ at $\bp$. More details about the precise definition of a tangent cone can be found in Appendix~\ref{SS:tangent} or \cite[Section 3]{BoDaPo14}. Let us now describe the nature of $\Pi_{\bp}$ according to the location of $\bp$ in the section of $\overline{\cone}$: 
\begin{enumerate}[(a)]
\item If $\bp$ belongs to $\cone\cap \dS^{2}$, i.e. is an interior point, then $\Pi_{\bp}=\R^{3}$.  
\item If $\bp$ belongs to the regular part of the boundary of $\cone \cap \dS^{2}$ (that is if $\bp$ is in the interior of a side of $\cone \cap \dS^{2}$), then $\Pi_{\bp}$ is a half-space.
\item If $\bp$ is a vertex of $\cone \cap \dS^{2}$ of opening $\theta$, then $\Pi_{\bp}$ is a wedge of opening $\theta$.
\end{enumerate}
The cone $\Pi_{\bp}$ is called a tangent substructure of $\cone$. The ground state energy of the magnetic Laplacian on $\Pi_{\bp}$ with magnetic field $\bB$ is well defined and still denoted by $\En(\bB,\Pi_{\bp})$.
Let us introduce the infimum of the ground state energies on the tangent substructures of $\cone$:
\begin{equation}\label{eq.Eness}
   \seE(\bB,\cone):=\inf_{\bp\,\in\, \overline{\cone} \cap \dS^{2}}\En(\bB,\Pi_{\bp}).
\end{equation}
Then \cite[Theorem 6.6]{BoDaPo14} yields that the bottom of the essential spectrum $\En_{\ess}(\bB,\cone)$ of the operator $\OP(\bA,\cone)$ is given by  this quantity:
\begin{equation}
\label{eq.Eness2}
   \En_{\ess}(\bB,\cone)=\seE(\bB,\cone).
\end{equation}
Now we take the view point of small angle asymptotics, like in \eqref{eq:sector}, \eqref{eq:conecirc}, and \eqref{eq:wedge}. But for general 3D cones there is no obvious notion of small angle $\alpha$. That is why we introduce families of sharp cones for which the plane section $\omega$ is scaled by a small parameter $\a>0$. More precisely, $\omega\subset \R^{2}$ being given, we define the dilated domain
\begin{equation}\label{E:defweps}
\omega_{\a}:=\a \omega, \quad \a>0,
\end{equation}
and consider the family of cones $\conea$ parametrized by \eqref{E:defweps}, as $\a\to 0$. The homogeneity \eqref{eq:homo} of the bound $e(\bB,\omega)$ implies immediately
\begin{equation}
\label{eq:homoa}
   e(\bB,\omega_{\a}) = e(\bB,\omega)\,\a\,.
\end{equation}
Thus the bound \eqref{eq:Ee} implies that the Rayleigh quotients $\En_{n}(\bB,\conea)$ tend to $0$ as $\a\to0$.

To determine the asymptotic behavior of $\En_{\ess}(\bB,\conea)$ as $\a\to0$, we introduce $\widehat{\omega}$ as the cylinder $\omega\times \R$ and define the infimum of ground energies
$$
   \sE(\bB,\widehat{\omega})=\inf_{\bx'\in \overline{\omega}}\En(\bB,\widehat{\Pi}_{(\bx',1)}),
$$
where, for $\bx$ in the closure of ${\widehat{\omega}}$, $\widehat{\Pi}_\bx$ denotes the tangent cone to $\widehat{\omega}$ at $\bx$.  
We note that, by translation invariance along the third coordinate, $\sE(\bB,\widehat{\omega})$ is also the infimum of ground energies when $\bx$ varies in the whole cylinder $\overline{\widehat{\omega}}$.

\begin{prop}
\label{P:liminfspecess}
 Let $\omega$ be a bounded polygonal domain of $\R^2$, and $\omega_{\a}$ defined by \eqref{E:defweps}.
 Then
$$\lim_{\a\to0}\En_{\ess}(\bB,\conea)= \sE(\bB,\widehat{\omega})>0.$$
\end{prop}

Taking \eqref{eq:homoa} into account, as a direct consequence of Theorem \ref{prop.majo1} and Proposition \ref{P:liminfspecess}, we deduce:
\begin{cor}\label{cor.majo1}
Let $\omega$ be a bounded polygonal domain of $\R^2$ and $\bB$ be a constant magnetic field. 
For all $n\geq1$, for all $\a>0$, there holds
$$\En_{n}(\bB,\conea)\leq (4n-1) e(\bB,\omega)\a.$$
In particular, for $\a$ small enough, there exists an eigenvalue below the essential spectrum.
\end{cor}

\begin{remark}
It is far from being clear whether $(4n-1) e(\bB,\omega)\a$ can be the first term of an eigenvalue asymptotics, like this is the case for circular cones as proved in \cite{BoRay13,BoRay14}. 
\end{remark}

\subsubsection{Corner concentration in the semiclassical framework}
Let $\Omega\subset\R^3$ be a bounded simply connected corner domain in the sense of Definition~\ref{D:cornerdomains} (see \cite[Section 3]{BoDaPo14} for more details).
We denote by $\OP_{h}(\bA,\Omega)$ the Neumann realization of the Schr\"odinger operator $(-ih\nabla+\bA)^2$ on $\Omega$ with magnetic potential $\bA$ and semiclassical parameter $h$. 
Due to gauge invariance, its eigenvalues depend on the magnetic field $\bB=\curl\bA$, and not on the potential $\bA$, whereas the eigenfunctions do depend on $\bA$. We are interested in the first eigenvalue $\lambda_{h}(\bB,\Omega)$ of $\OP_{h}(\bA,\Omega)$ and in associated normalized eigenvector $\psi_{h}(\bA,\Omega)$. 

Let us briefly recall some of the results of \cite{BoDaPo14}, restricting the discussion to the case when the {\em magnetic field $\bB$ is constant} (and $\bA$ linear) for simplicity of exposition.
To each point $\bx\in\overline\Omega$ is associated with a dilation invariant, tangent open set $\Pi_\bx$, according to the following cases:
\begin{enumerate}
\item If $\bx$ is an interior point, $\Pi_\bx=\R^3$,
\item If $\bx$ belongs to a {\em face $\bff$} ({\it i.e.}, a connected component of the smooth part of $\partial\Omega$), $\Pi_\bx$ is a half-space,
\item If $\bx$ belongs to an {\em edge} $\be$, $\Pi_\bx$ is an infinite wedge,
\item If $\bx$ is a {\em vertex} $\bv$, $\Pi_\bx$ is an infinite cone.
\end{enumerate}
The \emph{local energy} $\En(\bB,\Pi_\bx)$ at $\bx$ is defined as the ground energy of the tangent operator $\OP(\bA,\Pi_\bx)$ and the {\em lowest local energy} is written as 
\begin{equation}
\label{eq:s}
   \sE(\bB \ee,\Omega) := \inf_{\bx\in\overline\Omega} \En(\bB,\Pi_\bx).
\end{equation}
Then \cite[Theorem 5.1 \& 9.1]{BoDaPo14} provides the general asymptotical bounds
\begin{equation}
\label{eq:11/10}
   |\lambda_{h}(\bB,\Omega) - h\sE(\bB,\Omega)|\leq 
   C \,h^{11/10} \quad \mbox{ as }\quad h\to0\ .  
\end{equation}
Let $\En_{\ess}(\bB,\Pi_\bx)$ be the bottom of the essential spectrum of $\OP(\bA,\Pi_\bx)$.
If there exists a vertex $\bv$ of $\Omega$ such that 
\begin{equation}\label{eq.condition}
   \sE(\bB,\Omega)=\En(\bB,\Pi_{\bv})<\En_{\ess}(\bB,\Pi_{\bv}),
\end{equation}
then there holds the improved upper bound 
$   \lambda_{h}(\bB,\Omega) \leq 
   h\sE(\bB,\Omega)+C\, h^{3/2}|\log h| $, see \cite[Theorem 9.1 (d)]{BoDaPo14}.
Finally, if the lowest local energy is attained at vertices only, in the following strong sense (here $\gV$ is the set of vertices of $\Omega$)
\begin{equation}\label{eq.condbis}
   \sE(\bB,\Omega)<\inf_{\bx\in\overline\Omega\setminus\gV} \En(\bB,\Pi_{\bx}),
\end{equation}
the first eigenvalue $\lambda_{h}(\bB,\Omega)$ has an asymptotic expansion as $h\to0$ ensuring the improved bounds
\begin{equation}
\label{eq:3/2}
   |\lambda_{h}(\bB,\Omega) - h\sE(\bB,\Omega)|\leq 
   C \,h^{3/2} \quad\mbox{  as }\quad h\to0\ , 
\end{equation}
and, moreover, the corresponding eigenfunction concentrates near the vertices $\bv$ such that $\sE(\bB,\Omega)=\En(\bB,\Pi_{\bv})$. This is an immediate adaptation of \cite{BoDa06} to the 3D case, see \cite[Section 12.1]{BoDaPo14}. In this framework, our result is now

\begin{prop}
\label{prop:cornercon}
Let $\omega$ be a bounded polygonal domain of $\R^2$, and $\omega_{\a}$ defined by \eqref{E:defweps}. 

\begin{enumerate}[a)]
\item Let $\big(\Omega(\varepsilon)\big)_\varepsilon$ be a family of 3D corner domains such that
\begin{enumerate}[i)]
\item One of the vertices $\bv(\varepsilon)$ of $\Omega(\varepsilon)$ satisfies $\Pi_{\bv(\varepsilon)}=\conea$,
\item The edge openings $\alpha_\bx$ of all domains $\Omega(\varepsilon)$ satisfy the uniform bounds
\begin{equation}
\label{eq:edgebound}
   \beta_0 \le \alpha_{\bx} \le 2\pi-\beta_0,\quad 
   \forall\bx\ \mbox{edge point of}\ \Omega(\varepsilon),\ \forall\varepsilon>0,
\end{equation} 
with a positive constant $\beta_0$.
\end{enumerate}
Then condition \eqref{eq.condbis} is satisfied for $\varepsilon$ small enough.

\item Families $\big(\Omega(\varepsilon)\big)_\varepsilon$ satisfying the above assumptions i) and ii) do exist.
\end{enumerate}
\end{prop}

\subsubsection{Outline of the paper}
The paper is organized as follows: Sections~\ref{sec.majo}--\ref{sec.optim} are devoted to the proof of Theorem~\ref{prop.majo1}: To get an upper bound of $\En_{n}(\bB,\cone)$, we introduce in Section~\ref{sec.majo} a reduced operator on the half-line, depending on the chosen axis $\rx_3>0$, 
and introduce test functions for the reduced Rayleigh quotients. Then, in Section \ref{sec.optim}, we optimize the choice of the magnetic potential $\bA$ in order to minimize the reduced Rayleigh quotients. The obtained upper bounds are explicitly computed in some examples like discs and rectangles. In Section \ref{sec.esssp}, we focus on the essential spectrum for a sharp cone $\conea$ with polygonal section and prove Proposition \ref{P:specesssharp} that is a stronger form of Proposition~\ref{P:liminfspecess}. Section \ref{sec.appli} is devoted to the proof of Proposition \ref{prop:cornercon} that provides cases of corner concentration for the first eigenvectors of the semiclassical magnetic Laplacian. We conclude the paper in Section \ref{sec:rob} by a comparison with Robin problem.
Finally, for completeness, we recall in Appendix~\ref{SS:tangent} the recursive definition of corner domains.

\section{Upper bound for the first Rayleigh quotients using a 1D operator}\label{sec.majo}
The aim of the two following sections is to establish an upper bound of the $n$-th Rayleigh quotient $\En_{n}(\bB,\cone)$, valid for any domain $\omega$. 

For any constant magnetic potential $\bB$, we introduce the subspace
$$\cA(\bB)=\{\bA\in\mathcal{L}(\R^3) : \quad\partial_{\rx_{3}}\bA=0 \quad\mbox{and}\quad \nabla\times\bA=\bB\},$$
where $\mathcal{L}(\R^3)$ denotes the set of the endomorphisms of $\R^3$.
The set $\cA(\bB)$ is not empty and we can consider $\bA\in\cA(\bB)$. 
Let $\omega$ be a bounded polygonal domain. 
We evaluate now the quadratic form $q[\bA,\cone](\varphi)$ for functions $\varphi$ only depending on the $\rx_{3}$ variable. 
This leads to introduce a new quadratic form on some weighted Hilbert space. 

\begin{lem}\label{L:eval1d}
Let us introduce the weighted space 
$\sL^2_{\w}(\R_{+}):=\sL^2(\R_{+},\rx^2 \rd \rx)$ endowed with the norm $\|u\|_{\sL^2_{\w}(\R_{+})}:= \left(\int_{\R_{+}} |u(\rx)|^2 \rx^2 \rd \rx\right)^{1/2}$.
For any parameter $\lambda>0$, we define the quadratic form $\harm[\lambda]$ by
$$
   \harm[\lambda](u) = 
   \int_{\R_{+}}\left(|u'(\rx)|^2 +\lambda \rx^2|u(\rx)|^2\right)\rx^2 \rd \rx ,
$$ 
on the domain $\sB_{\w}(\R_{+}):=\{u\in \sL^2_{\w}(\R_{+}):\, \rx u\in \sL^2_{\w}(\R_{+}) , u'\in \sL^2_{\w}(\R_{+})\}$. \\
Let $\bA\in\cA(\bB)$ and $\varphi\in\sB_{\w}(\R_{+})$. Then the function $\cone\ni\bx\mapsto \varphi(\rx_{3})$, still denoted by $\varphi$, belongs to $\dom(q[\bA,\cone])$. Moreover there holds 
$$
   \frac{q[\bA,\cone](\varphi)}{\|\varphi\|^2_{\sL^2(\cone)}}=
   \frac{\harm\big[\lambda\big](\varphi)}{\|\varphi\|^2_{\sL^2_{\w}(\R_{+})}}
   \qquad \mbox{with}\qquad
   \lambda=\frac{\|\bA\|^2_{\sL^2(\omega)}}{|\omega|}.
$$
\end{lem}

\begin{proof}
Let $\bA = (\sfA_{1},\sfA_{2},\sfA_{3})^{\sf T}\in\cA(\bB)$. 
Since $\varphi$ is real valued and depends only on the $\rx_{3}$ variable, we have
\begin{eqnarray*}
q[\bA,\cone](\varphi)
&=& \int_{\cone}|\sfA_{1}|^2|\varphi|^2+|\sfA_{2}|^2|\varphi|^2+|(-i\partial_{\rx_{3}}+\sfA_{3})\varphi|^2 \rd \bx\\
&=& \int_{\cone}|\bA(\bx)|^2|\varphi(\rx_{3})|^2+|\partial_{\rx_{3}}\varphi(\rx_{3})|^2 \rd \bx.
\end{eqnarray*}
Let us perform the change of variables
\begin{equation}\label{eq.chgvar}
\bX = (\rX_{1},\rX_{2},\rX_{3}) = \left(\frac{\rx_{1}}{\rx_{3}},\frac{\rx_{2}}{\rx_{3}},\rx_{3}\right).
\end{equation}
Since $\bA$ is linear and does not depends on $\rx_{3}$, we have
\begin{align*}
q[\bA,\cone](\varphi)
&= \int_{\omega\times\R^+}\Big(|\bA(\bX)|^2 \rX_{3}^2 |\varphi(\rX_{3})|^2
+|\varphi'(\rX_{3})|^2\Big)\rX_{3}^2\rd \bX\\
&= |\omega| \int_{\R^+} |\varphi'(\rX_{3})|^2 \rX_{3}^2 \rd \rX_{3}
+\|\bA\|^2_{\sL^2(\omega)}\int_{\R^+}|\varphi(\rX_{3})|^2 \rX_{3}^4\rd \rX_{3},
\end{align*}
and, with the same change of variables~\eqref{eq.chgvar}
\begin{eqnarray*}
\|\varphi\|^2_{\sL^2(\cone)}
&=& |\omega|\int_{\R^+} |\varphi(\rX_{3})|^2 \rX_{3}^2 \rd \rX_{3}.
\end{eqnarray*}
Thus the Rayleigh quotient writes
\begin{eqnarray*}
\frac{q[\bA,\cone](\varphi)}{\|\varphi\|^2_{\sL^2(\cone)}}
&=& \frac{\int_{\R^+} |\varphi'(\rX_{3})|^2 \rX_{3}^2 \rd \rX_{3}
+ \frac{\|\bA\|^2_{\sL^2(\omega)}}{|\omega|}\int_{\R^+}|\varphi(\rX_{3})|^2 \rX_{3}^4\rd \rX_{3}}
{\int_{\R^+}|\varphi(\rX_{3})|^2 X_{3}^2 \rd \rX_{3}},
\end{eqnarray*}
and we deduce the lemma.
\end{proof}

With Lemma \ref{L:eval1d} at hands, we are interested in the spectrum of the operator associated with the quadratic form $\harm[\lambda]$. Thanks to the change of function $u\mapsto U:=\rx u$, the weight is eliminated and we find by using an integration by parts that
\[
   \harm[\lambda](u) = \int_{\R_{+}}\left(|U'(\rx)|^2 +\lambda \rx^2|U(\rx)|^2\right) \rd \rx
   \quad\mbox{and}\quad
   \|u\|^2_{\sL^2_{\w}(\R_{+})} = \|U\|^2_{\sL^2(\R_{+})}.
\]
So we are reduced to an harmonic oscillator on $\R_{+}$ with Dirichlet condition at $0$. Its eigenvectors $U_n$ are the restrictions to $\R_{+}$ of the odd ones on $\R$. Therefore, see also \cite[Corollary C.2]{BoRay13}, we find that the eigenvalues of the operator associated with the form $\harm[\lambda]$ are simple and the $n$-th eigenvalue equals $\lambda^{1/2}(4n-1)$. 
Then, by combining the min-max principle with Lemma \ref{L:eval1d}, we deduce that the $n$-th eigenvalue of the operator associated with the form $q[\bA,\cone]$ is bounded from above by 
$(4n-1)\|\bA\|_{\sL^2(\omega)}/\sqrt{|\omega|}$. 
Since this upper bound is valid for any $\bA\in\cA(\bB)$, we have proved the following proposition.

\begin{prop}\label{th.majon}
Let $\bB$ be a constant magnetic field. Then for all $n\in \N^{*}$,  
we have
\begin{equation}\label{eq.dlalpha}
\En_{n}(\bB,\cone)\leq \frac{4n-1}{\sqrt{|\omega|}}\inf_{\bA\in\cA(\bB)}\|\bA\|_{\sL^2(\omega)},
\end{equation}
 with 
 $$\cA(\bB)=\{\bA\in\mathcal{L}(\R^3) : \quad\partial_{\rx_{3}}\bA=0 \quad\mbox{ and } \quad\nabla\times\bA=\bB\}.$$
 \end{prop}

\section{Optimization}\label{sec.optim}
The aim of this section is to give an explicit solution to the optimization problem 
\begin{equation}\label{eq.pbopt3D}
 \mbox{\sl Find \ $\bA_0\in\cA(\bB)$ \ such that}\quad 
   \|\bA_0\|_{\sL^2(\omega)} = \inf_{\bA\in\cA(\bB)}\|\bA\|_{\sL^2(\omega)},
\end{equation}
for a constant magnetic field $\bB = (\sfB_{1},\sfB_{2}, \sfB_{3})^{\sf T}$. We also provide explicit examples  in the case where the domain $\omega$ is a disc or a rectangle. 
\subsection{Resolution of the optimization problem and proof of Theorem \ref{prop.majo1}}
Let $\bA= (\sfA_{1},\sfA_{2},\sf A_{3})^{\sf T}\in \cA(\bB)$. Since $\bA$ is independent of the $\rx_{3}$ variable, we have
$$\curl \bA = \begin{pmatrix}\partial_{\rx_{2}} \sfA_{3}\\ -\partial_{\rx_{1}} \sfA_{3}\\ \partial_{\rx_{1}}\sfA_{2}-\partial_{\rx_{2}} \sfA_{1}\end{pmatrix}
=\begin{pmatrix} \sfB_{1}\\ \sfB_{2}\\ \sfB_{3}\end{pmatrix}.$$
By linearity of $\bA$, we have necessarily $\sfA_{3}(\bx)=\sfB_{1} \rx_{2}-\sfB_{2}\rx_{1}$. Therefore considering
$$\cA' = \{\bA'\in\mathcal{L}(\R^2) : \quad\nabla_{\rx_{1},\rx_{2}}\times\bA'=1\},$$
the infimum in \eqref{eq.pbopt3D} rewrites
\begin{equation}
\label{E:decomposinf}
\inf_{\bA\in\cA(\bB)}\|\bA\|_{\sL^2(\omega)}=
\left(\sfB_{3}^2\inf_{\bA'\in\cA'} \|\bA'\|^2_{\sL^2(\omega)} + \int_{\omega}(\sfB_{1} \rx_{2}-\sfB_{2}\rx_{1})^2\rd \rx_{1}\rd \rx_{2}\right)^{1/2},
\end{equation}
and 3D optimization problem \eqref{eq.pbopt3D} can be reduced to a 2D one:
\begin{equation}
\label{E:Minim2d}
\mbox{\sl Find \ $\bA'_0\in\cA'$ \ such that}\quad 
   \|\bA'_0\|_{\sL^2(\omega)} =
   \inf_{\bA'\in\cA'} \|\bA'\|_{\sL^2(\omega)}.
\end{equation}
This problem can be solved explicitly:
\begin{prop}
\label{P:minim2d}
For $k=0,1,2$, we define the moments
$$M_{k}:=\int_{\omega}\rx_{1}^{k}\rx_{2}^{2-k}\rd \rx_{1} \rd \rx_{2}.$$
Then, we have
$$\inf_{\bA'\in\cA'} \|\bA'\|_{\sL^2(\omega)}^2
=\frac{M_{0}M_{2}-M_{1}^2}{M_{0}+M_{2}}.$$
Moreover the minimizer of \eqref{E:Minim2d} exists, is unique, and given by 
$$
 \bA'_{0}(\rx_{1},\rx_{2})=\frac{1}{M_{0}+M_{2}}\begin{pmatrix} M_{1} & -M_{0} \\ M_{2} & -M_{1} \end{pmatrix}\begin{pmatrix} \rx_{1} \\ \rx_{2} \end{pmatrix}.
$$
\end{prop}

\begin{remark}
a) Let us notice that 
$$M_{0} M_{2}-M_{1}^2
= \frac 12 \int_{\omega} \int_{\omega} (\rx_{1} \rx_{2}'-\rx_{1}' \rx_{2})^2 \rd\rx_{1} \rd\rx_{2} \rd\rx_{1}' \rd\rx_{2}'. 
$$ 
This relation highlights once more the connection with he geometry of $\omega$. 

b) The divergence of the optimal transverse potential $\bA'_0$ is $0$, just as the full associated potential $\bA_0$.
\end{remark}

\begin{proof}
Let us introduce the space of linear applications of the plane $\cL(\R^{2})$ endowed with the scalar product
$$\langle f,g \rangle_{\sL^2(\omega)}=\int_{\omega} f(\rx_{1},\rx_{2})\cdot g(\rx_{1},\rx_{2}) \rd \rx_{1}\rd\rx_{2},\qquad \forall f,g\in \cL(\R^{2}).$$
Then $\cA'$ is an affine hyperplane of $\cL(\R^{2})$ of dimension 3, and Problem \eqref{E:Minim2d} is equivalent to find the distance from the origin $\bfz$ to this hyperplane. In particular there exists a unique minimizer
to \eqref{E:Minim2d}, which is the orthogonal projection of $\bfz$ to $\cA'$.
To make the solution explicit, we look for a linear function $\bA'_{0}\in \cA'$ of the form 
$$\bA'_{0} (\rx_{1},\rx_{2})
=\begin{pmatrix} \alpha & \beta \\ 1+\beta & \gamma \end{pmatrix}\begin{pmatrix} \rx_{1} \\ \rx_{2} \end{pmatrix},
$$
where $(\alpha,\beta,\gamma)$ are to be found. 
Then we have
\begin{align*}
F(\alpha,\beta,\gamma) &:= \|\bA'_{0}\|^2_{\sL^2(\omega)}
= \int_{\omega} (\alpha \rx_{1}+\beta \rx_{2})^2+((1+\beta) \rx_{1}+\gamma \rx_{2})^2 \rd \rx_{1} \rd \rx_{2}\\
&= M_{2}(\alpha^2+(1+\beta)^2) + 2M_{1}(\alpha\beta+(1+\beta)\gamma) + M_{0}(\beta^2+\gamma^2).
\end{align*}
Solving $\nabla F = 0$ gives a unique solution
$$
(\alpha,\beta,\gamma) = \frac{1}{M_{0}+M_{2}}(M_{1}, -M_{0} ,-M_{1}),
$$
and computations provide
$$\|\bA'_{0}\|^2_{\sL^2(\omega)}= \frac{M_{0}M_{2}-M_{1}^2}{M_{0}+M_{2}}.$$
We deduce the proposition.
\end{proof}
\begin{proofof}{Theorem~\ref{prop.majo1}}
Now, combining Proposition \ref{th.majon}, \eqref{E:decomposinf} and Proposition \ref{P:minim2d}, we get the upper bound 
$$\En_{n}(\bB,\cone)\leq (4n-1)e(\bB,\cone),$$
with
\begin{align*}
e(\bB,\omega)
   &=\frac{1}{\sqrt{|\omega|}}\left(\sfB_{3}^2\frac{M_{0}M_{2}-M_{1}^2}{M_{0}+M_{2}}
   +\int_\omega (\rx_{1}\sfB_{2}-\sfB_{2}\rx_{1})^2 \rd \rx_{1} \rd \rx_{2} \right)^{1/2} \\ 
   & =\frac{1}{\sqrt{|\omega|}}\left(\sfB_{3}^2\frac{M_{0}M_{2}-M_{1}^2}{M_{0}+M_{2}}
   + \sfB_{2}^2M_{2} + \sfB_{1}^2M_{0} -2 \sfB_{1}\sfB_{2}M_{1} \right)^{1/2} \\
   & =\left(\sfB_{3}^2\frac{m_{0}m_{2}-m_{1}^2}{m_{0}+m_{2}}
   + \sfB_{2}^2m_{2} + \sfB_{1}^2m_{0} -2 \sfB_{1}\sfB_{2}m_{1} \right)^{1/2},
\end{align*}
with $m_{k}=M_{k}/|\omega|$, and we deduce Theorem \ref{prop.majo1}.
\end{proofof}

\begin{proofof}{Lemma~\ref{lem.homog}}
Let us discuss the quantities appearing in $e(\bB,\omega)$:
\begin{itemize}
\item The coefficient $m_{0}m_{2}-m_{1}^2$ corresponds to a Gram determinant, and is positive by Cauchy-Schwarz inequality. 
\item The coefficient $m_{0}+m_{2}=\frac{1}{|\omega|}\int_{\omega}(\rx_{1}^2+\rx_{2}^2) \rd \rx_{1} \rd \rx_{2}$ is the isotropic moment of order 2 in $\omega$.
\item When $(\sfB_{1},\sfB_{2})\neq0$, we denote by $\Delta\subset \R^2$ the line borne by the projection of the magnetic field in the plane $\{\rx_{3}=0\}$. Then the quantity 
$$\int_\omega (\sfB_{2}\rx_{1}-\sfB_{1}\rx_{2})^2 \rd \rx_{1} \rd \rx_{2}$$
 is the square of the $\sL^2$ norm (in $\omega$) of the distance to $\Delta$.
\end{itemize}
Consequently, the function $\bB\mapsto e(\bB,\omega)$ is a norm on $\R^{3}$. Furthermore, although the normalized moments depend on the choice of Cartesian coordinates in $\R^2$, the above three points show that this is not the case for the three quantities $m_{0}+m_{2}$, $m_{2}m_{0}-m_{1}^2$ and $b_{2}^2m_{2} + b_{1}^2m_{0} -2 b_{1}b_{2}m_{1}$. We deduce that the constant $e(\bB,\omega)$ depends only on the magnetic field and the domain and not on the choice of Cartesian coordinates. Lemma \ref{lem.homog} is proved.
\end{proofof}

\subsection{Examples}
In this section we apply Proposition \ref{P:minim2d} to particular geometries, namely discs and rectangles.
\subsubsection{Circular cone}\label{subsec.circular}
The case of a right circular cone is already considered in \cite{BoRay13,BoRay14}, and we compare our upper bound given in Theorem~\ref{prop.majo1} with the existing results.

For any disc $\omega$ centered at the origin, the normalized moments equal
$$
m_{0}=m_{2}=\frac{|\omega|}{4\pi} \quad \mbox{and} \quad m_{1}=0,  
$$
so that  
Theorem~\ref{prop.majo1} gives
\begin{equation} \label{eq:disc}
\En_{n}(\bB,\cone)\leq
(4n-1)e(\bB,\omega)
= \frac{4n-1}{2}\sqrt{\frac{|\omega|}\pi}\left( \frac{\sfB_{3}^2}2 +\sfB_{1}^2 + \sfB_{2}^2 \right)^{1/2}.
\end{equation}
In \cite{BoRay13,BoRay14}, the right circular cone $\cC^\circ_\alpha$ with opening $\alpha$ is considered: Here $\omega$ is the disc centered at the origin with radius $\tan\frac{\alpha}{2}$. In this case, a complete asymptotic expansion is established as $\alpha\to 0$ and the first term is given by
\begin{equation}\label{eq.lima}
\lim_{\alpha\to 0}\frac{\En_{n}(\bB,\cC^\circ_\alpha)}{\alpha} = \frac{4n-1}{2^{5/2}}\sqrt{1+\sin^2\beta},
\end{equation}
where $\beta$ is the angle between the magnetic field $\bB$ and the axis of the cone.
Let us compare with our upper bound \eqref{eq:disc}, applied with $\bB= (0, \sin \beta,\cos\beta)^{\sf T}$ and $|\omega|=\pi\tan^2\frac{\alpha}{2}$. This provides:
$$
 \forall \alpha\in (0,\pi), \quad \En_{n}(\bB,\cC^\circ_\alpha)\leq \frac{4n-1}{2^{3/2}}\tan\frac{\alpha}{2} \sqrt{1+\sin^2\beta}.
$$
In view of \eqref{eq.lima}, this upper bound is optimal asymptotically, as $\alpha\to0$.
Let us notice that the solution of the minimization problem \eqref{E:Minim2d} is in that case the so called symmetric potential
 $\bA'_{0} = \frac12 \left(-\rx_{2},\rx_{1}\right)^{\sf T}$ (see Proposition \ref{P:minim2d}).

\subsubsection{Rectangular cone}
Let us assume that $\omega$ is the rectangle $[\ell_{a},\ell_{b}]\times[L_{a},L_{b}]$. \\
The moments of order 2 can be computed explicitly:
\begin{align*}
m_{0} &=\frac{(\ell_{b}-\ell_{a})(L_{b}^3-L_{a}^3)}{3|\omega|} = \frac 13(L_{b}^2+L_{b}L_{a}+L_{a}^2),\\ 
m_{1} &= \frac{(\ell_{b}^2-\ell_{a}^2)(L_{b}^2-L_{a}^2)}{4|\omega|} = \frac 14(\ell_{b}+\ell_{a})(L_{b}+L_{a}),\\
m_{2} &=\frac{(\ell_{b}^3-\ell_{a}^3)(L_{b}-L_{a})}{3|\omega|} = \frac 13(\ell_{b}^2+\ell_{b}\ell_{a}+\ell_{a}^2).
\end{align*}
Let us apply Theorem~\ref{prop.majo1} in several configurations. 
Note that if $\ell_{a}=-\ell_{b}$ or $L_{a}=-L_{b}$ (which means that we have a symmetry), then $m_{1}=0$ and 
$$\En_{n}(\bB,\cone)\leq 
{(4n-1)} \left(\sfB_{3}^2\frac{m_{0}m_{2}}{m_{0}+m_{2}} +\sfB_{1}^2 m_{0} + \sfB_{2}^2 m_{2}\right)^{1/2}.$$
Assuming, both $\ell_{a}=-\ell_{b}$ and $L_{a}=-L_{b}$, we obtain the following upper bound for the ground state energy for the rectangle $[-\ell,\ell]\times[-L,L]$ (for shortness, $\ell=\ell_{b}$ and $L=L_{b}$):
\begin{equation}\label{eq.rectlL}
   \En_{n}(\bB,\cone)\leq  \frac{4n-1}{\sqrt 3}   \left(\sfB_{3}^2\frac{\ell^2 L^2}{\ell^2 + L^2} +\sfB_{1}^2 L^2 + \sfB_{2}^2 \ell^2 \right)^{1/2}.
\end{equation}
In the case of a symmetric rectangle of proportions $\ell<L=1$, the last formula becomes
\[
   \En_{n}(\bB,\cone)\leq  
    \frac{4n-1}{\sqrt 3}\left(\sfB_{3}^2\frac{\ell^2}{\ell^2 + 1} + \sfB_{1}^2 + \sfB_{2}^2 \ell^2 \right)^{1/2}.
\]
We observe that this upper bound does not converge to 0 when $\sfB_{1}\neq0$ and $\ell$ tends to $0$. In contrast when $\sfB_{1}=0$ there holds
\[
   \En_{n}(\bB,\cone) \leq \frac{4n-1}{\sqrt 3} \,\ell\,
   \left(\frac{\sfB_{3}^2}{\ell^2 + 1} + \sfB_{2}^2 \right)^{1/2},
\]
which tends to $0$ as $\ell\to0$. This configuration ($\sfB_{1}=0$ and $\ell\to 0$) means that $\bB$ is almost tangent to the cone $\cone$ in the direction where it is not sharp. This can be compared with the result \eqref{eq:wedge} on wedges.
This shows the anisotropy of the quantities appearing in our upper bounds.

For the square $[-\ell,\ell]^2$, we deduce the upper bound of the first eigenvalue
\begin{equation}\label{eq:square}
\En_{n}(\bB,\cone)
\leq \frac{4n-1}{\sqrt 3}\ell \left(\frac{\sfB_{3}^2}2 +\sfB_{1}^2 + \sfB_{2}^2 \right)^{1/2}  \!\!\!
=\frac{4n-1}{2} \frac{\sqrt{|\omega|}}{\sqrt 3}\left( \frac{\sfB_{3}^2}2 +\sfB_{1}^2 + \sfB_{2}^2 \right)^{1/2}\!\!\!.
\end{equation}
\begin{rem}
Assuming that $|\omega|$ is set, our upper bounds in the case when $\omega$ is a square or a disc can be compared, see \eqref{eq:disc} 
and \eqref{eq:square}. The distinct factors are
\[
	\frac{1}{\sqrt\pi} \simeq 0.5642\quad\mbox{and} \quad
	\frac{1}{\sqrt 3} \simeq 0.5774.
\]
\end{rem}

%%%%%%%%%%%%%%%%%%%%%%%%%%%%
\section{Essential spectrum for cones of small apertures with polygonal section}\label{sec.esssp}
%%%%%%%%%%%%%%%%%%%%%%%%%%%%
Here we consider the case of a family of cones parametrized by a model plane polygonal domain $\omega\subset \R^2$ and the scaling factor $\a>0$. We characterize the limit of the bottom of the essential spectrum $\En_{\ess}(\bB,\conea)$ as $\a\to0$, where $\conea$ is defined in \eqref{E:defweps}. The main result of this section is Proposition \ref{P:specesssharp}, which is a stronger version of Proposition \ref{P:liminfspecess}.

In such a situation, relations \eqref{eq.Eness}--\eqref{eq.Eness2} take the form
$$
   \En_{\ess}(\bB,\conea)=\seE(\bB,\conea)=\inf_{\bp\,\in\,\overline{\conea} \cap \dS^{2}}\En(\bB,\Pi_{\bp}).
$$
We define the bijective transformation $\rP: \omega\times\R_+ \to \cC_{\omega}$  by 
\begin{equation}
\label{D:projrP}
\rP(\bx',t)= t\,\frac{(\bx',1)}{\|(\bx',1)\|},\qquad\forall (\bx',t)\in\omega\times\R_+.  
\end{equation}
Notice that $\bx'\mapsto \rP(\bx',1)$ defines a bijection from $\R^{2}$ onto the upper half sphere $\dS^{2}_{+}:=\{\bp \in \dS^{2},\ {\mathsf p}_{3}>0\}$, and that for all $\a>0$, $\rP(\a\omega,1)$ is an open set of $\dS^{2}_{+}$ and coincides with $\conea \cap \dS^{2}$.

If $\bp$ is a vertex of $\conea \cap \dS^{2}$, then $\bx'=\rP(\cdot,1)^{-1}(\bp)$ is still a vertex of $\omega_{\a}$, but its opening angle is not the same as for $\bp$, in particular the tangent cones $\Pi_{\bp}$ and $\widehat{\Pi}_{\bx'}$ are both wedges, but they cannot be deduced each one from another by a rotation, and in general the ground state energies on these two domains are different.

The following proposition estimates the difference between the ground state energies as $\a \to 0$:

\begin{prop}
\label{P:specesssharp}
There exist positive constants $\a_0$ and $C(\omega)$ depending only on $\omega$ such that
\begin{equation}
\label{E:specesssharp}
   \forall\a\in(0,\a_{0}),\qquad
   |\seE(\bB,\conea)-\sE(\bB,\widehat{\omega})| \le C(\omega)\,\a^{1/3}.
\end{equation}
In particular, $\lim_{\a\to0} \seE(\bB,\conea)=\sE(\bB,\widehat{\omega})$.
\end{prop}

\begin{proof}
Recall that the transformation $\rP$ is defined in \eqref{D:projrP}. Denote by $\bf{0}$ the origin in the plane $\R^2$. The differential $\rd_{{(\bf{0}},1)} \rP$ of $\rP$ at the point $({\bf{0}},1)$ is the identity $\Id$. So there exist positive constants $C$ and $\a_{0}$ such that for all $\a\in (0,\a_{0})$, 
\begin{equation}
\label{eq:P}
  \forall \bx'\in \overline{\omega_{\a}},\quad \|\rd_{(\bx',1)}\rP-\Id\| \leq C\a.
\end{equation}  
Define $\rN_{\a}$ the scaling of ratio $\a$ around the plane $t=1$:
\begin{equation}\label{D:Neps}
   \rN_{\a}:(\rx_{1},\rx_{2},t)\longmapsto (\a \rx_{1},\a \rx_{2},1+\a(t-1)).
\end{equation}
The scaling $\rN_{\a}$ transforms a neighborhood of $\overline\omega\times\{1\}$ into a neighborhood of $\a\overline\omega\times\{1\}$. Then the composed application $\rP\circ \rN_{\a}$ is a diffeomorphism from a neighborhood of $\overline\omega\times\{1\}$ onto a neighborhood of $\conea\cap\dS^{2}$.

Let us pick a point $\bx'$ in the closure of the polygonal domain $\omega$. By definition of polygonal domains, there exists a local diffeomorphism $\rJ$ that sends a neighborhood of $\bx'$ in $\overline\omega$ onto a neighborhood of $\bf 0$ of the tangent plane sector (in broad sense) $\Pi_{\bx'}$. The differential $\rd_{\bx'} \rJ$ equals $\Id$ by construction. Then $\widehat\rJ := \rJ\otimes\Id_{3}$ realizes a local diffeomorphism that sends a neighborhood of $\bx:=(\bx',1)$ in $\overline{\widehat\omega}$ onto a neighborhood of $\bf 0$ of the tangent cone $\widehat\Pi_{\bx}:= \Pi_{\bx'}\times\R$.

We set $\bp_{\a}:=\rP\circ \rN_{\a}(\bx)$. For any $\a\in(0,\a_{0})$, the composed application 
$$\widehat\rJ\circ(\rP\circ \rN_{\a})^{-1}$$ 
is a local diffeomorphism that sends a neighborhood of the point $\bp_{\a}$ in $\overline\cC_{\a\omega}$ onto a neighborhood of $\bf 0$ of the cone $\widehat\Pi_{\bx}$. Let $\rD_{\a}$ be the differential at $\bf 0$ of the inverse of the map $\widehat\rJ\circ(\rP\circ \rN_{\a})^{-1}$. Then, by construction, the modified map
\[
   \rD_{\a} \circ \widehat\rJ\circ(\rP\circ \rN_{\a})^{-1}
\]
is such that its differential at the point $\bp_{\a}$ is the identity $\Id$. 
Therefore this modified map is a local diffeomorphism that sends a neighborhood of the point $\bp_{\a}$ in $\overline{\conea}$ onto a neighborhood of ${\bf 0}$ in the \emph{tangent} cone $\Pi_{\bp_{\a}}$. 

We deduce that $\rD_{\a}$ is a linear isomorphism between the two cones of interest
\[
   \rD_{\a} : \widehat\Pi_{\bx} \longmapsto \Pi_{\bp_{\a}}\,.
\]
We calculate:
\[
   \rD_{\a} = \rd_{\bf 0} (\rP\circ \rN_{\a}\circ\widehat\rJ^{-1})
   = \rd_{\bp_{\a}}\rP\circ \rd_{\bx}\rN_{\a}\circ \rd_{0}\widehat\rJ^{-1}\,.
\]
But $\rd_{\bf 0}\widehat\rJ^{-1}=\Id$ and $\rd_{\bx}\rN_{\a}=\a\Id$. So we have obtained that $\a\rd_{\bp_{\a}}\rP$ is an isomorphism between the two cones of interest. By homogeneity $\rd_{\bp_{\a}}\rP$ is also an isomorphism between the same sets. Thanks to \eqref{eq:P} we have obtained that

\begin{lem}
Let $\bx'\in\overline\omega$, $\bx=(\bx',1)$ and $\bp_{\a}=\rP\circ \rN_{\a}(\bx)$. Then the linear map $\rL_{\bx,\a} := \rd_{\bp_{\a}}\rP$ is an isomorphism between $\widehat\Pi_{\bx}$ and $\Pi_{\bp_{\a}}$, that satisfies
\begin{equation}
\label{eq:Px}
   \|\rL_{\bx,\a} - \Id\| \leq C\a,
\end{equation}
where $C$ depends neither on $\bx'$ nor on $\a$ and with $\rP, \rN_{\a}$ defined in \eqref{D:projrP}, \eqref{D:Neps}.
\end{lem}
Therefore
\begin{equation}
\label{eq;Encomp}
   \En(\bB,\widehat{\Pi}_{\bx}) - \En(\bB,\Pi_{\bp_{\a}}) =
   \En(\bB,\widehat{\Pi}_{\bx}) - \En(\bB,\rL_{\bx,\a}(\widehat{\Pi}_{\bx})) .
\end{equation}
Relying on \eqref{eq:Px}, we are going to estimate the right hand side of \eqref{eq;Encomp} depending on the position of $\bx'\in\overline{\omega}$:

(a) $\bx'$ is inside $\omega$. Then  $\widehat{\Pi}_{\bx}$ is the full space $\R^3$, just like $\rL_{\bx,\a}(\widehat{\Pi}_{\bx})$. So $\En(\bB,\widehat{\Pi}_{\bx})$ coincides with $\En(\bB,\rL_{\bx,\a}(\widehat{\Pi}_{\bx}))$ in this case.

(b) $\bx'$ belongs to a side of $\omega$. Then $\widehat{\Pi}_{\bx}$ and $\rL_{\bx,\a}(\widehat{\Pi}_{\bx})$ are half-spaces. The lowest energy $\En(\bB,\Pi)$ when $\Pi$ is a half-space is determined by the $\sC^1$ function $\sigma$ acting on the unsigned angle $\theta \in [0,\frac{\pi}{2}]$ between $\bB$ and $\partial\Pi$. If $\theta_{\bx}$\,, $\theta_{\bx,\a}$ denote the angle between $\bB$ and $\partial\widehat{\Pi}_{\bx}$\,, $\partial\rL_{\bx,\a}(\widehat{\Pi}_{\bx,\a})$, respectively, then for a constant $C$ depending on $\omega$:
\begin{equation}
\label{eq.sig}
   |\theta_{\bx}-\theta_{\bx,\a}|\le C\a \quad\mbox{and}\quad
   |\sigma(\theta_{\bx}) - \sigma(\theta_{\bx,\a})| \le C\a.
\end{equation}

(c) $\bx'$ is a corner of $\omega$. Then $\widehat{\Pi}_{\bx}$ and $\rL_{\bx,\a}(\widehat{\Pi}_{\bx})$ are wedges of opening $\alpha_{\bx}$ and $\alpha_{\bx,\a}$ with $|\alpha_{\bx}-\alpha_{\bx,\a}|\le C\a$. Moreover there exist rotations $\rR_{\bx}$ and $\rR_{\bx,\a}$ that transform $\widehat{\Pi}_{\bx}$ and $\rL_{\bx,\a}(\widehat{\Pi}_{\bx})$ into the canonical wedges $\cW_{\alpha_{\bx}}$ and $\cW_{\alpha_{\bx,\a}}$ and there holds $\|\rR_{\bx,\a}-\rR_{\bx}\| \leq C\a$. Since
\[
   \En(\bB,\widehat{\Pi}_{\bx}) = \En(\rR_{\bx}^{-1}\bB,\cW_{\alpha_{\bx}})
   \quad\mbox{and}\quad
   \En(\bB,\rL_{\bx,\a}(\widehat{\Pi}_{\bx})) = \En(\rR_{\bx,\a}^{-1}\bB,\cW_{\alpha_{\bx,\a}}),
\]
we deduce from \cite[Section 4.4]{Pop13}
\[
   |\En(\bB,\widehat{\Pi}_{\bx})-\En(\bB,\rL_{\bx,\a}(\widehat{\Pi}_{\bx}))| \le C\a^{1/3}.
\]
Taking the infimum over $\bx\in \overline{\omega}\times\{1\}$, we deduce the \eqref{E:specesssharp}. As stated in \cite[Corollary 8.5]{BoDaPo14}, there holds $\sE(\bB,\widehat{\omega})>0$. Therefore we deduce Proposition \ref{P:specesssharp}.
\end{proof}

\section{Application to corner concentration\label{sec.appli}}
In this section, we discuss the link between \eqref{eq.condition} and \eqref{eq.condbis}, and we then prove Proposition \ref{prop:cornercon}.

We first prove that condition \eqref{eq.condbis} implies condition \eqref{eq.condition}.
If \eqref{eq.condbis} holds, there exists a vertex $\bv$ such that $\sE(\bB,\Omega)=\En(\bB,\Pi_{\bv})$.
By \cite[Theorem 6.6]{BoDaPo14}, the essential spectrum of $\OP(\bA,\Pi_\bv)$ is given by 
\[
   \seE(\bB,\Pi_\bv):=\inf_{\bp\,\in\, \overline\Pi_\bv \cap \dS^{2}}\En(\bB,\Pi_{\bp}).
\]
But for each $\bp\,\in\, \overline\Pi_\bv \cap \dS^{2}$, the cone $\Pi_\bv$ is the limit of tangent cones $\Pi_\bx$ with points $\bx\in\overline\Omega\setminus\gV$ converging to $\bv$. The continuity of the ground energy then implies that 
\[
   \En(\bB,\Pi_{\bp}) \ge \inf_{\bx\in\overline\Omega\setminus\gV} \En(\bB,\Pi_{\bx}).
\]
We deduce
\[
   \seE(\bB,\Pi_\bv)\ge \inf_{\bx\in\overline\Omega\setminus\gV} \En(\bB,\Pi_{\bx}).
\]
Hence condition \eqref{eq.condition} holds.

Proof of point \emph{a)} of Proposition \ref{prop:cornercon}. By condition \emph{i)}, and as a consequence of \eqref{eq:Ee} and \eqref{eq:homoa}, there holds
\begin{equation}
\label{eq:Env}
   \En(\bB,\Pi_{\bv(\varepsilon)}) \le 3\varepsilon\, e(\bB,\omega).
\end{equation}
Let us bound $\inf_{\bx\in\overline\Omega\setminus\gV} \En(\bB,\Pi_{\bx})$ from below. Let $\bx\in\overline\Omega\setminus\gV$.
\begin{enumerate}
\item If $\bx$ is an interior point, then $\En(\bB,\Pi_{\bx}) = \En(\bB,\R^3)=\|\bB\|$.
\item If $\bx$ belongs to a face, $\Pi_{\bx}$ is a half-space and $\En(\bB,\Pi_{\bx}) \ge \Theta_0\|\bB\| > \frac12 \|\bB\|$.
\item Since $\bx$ is not a vertex, it remains the case when $\bx$ belongs to an edge of $\Omega$, and then $\Pi_{\bx}$ is a wedge. Let $\alpha_\bx$ denote its opening. Then $\En(\bB,\Pi_{\bx})=\En(\bB_{\bx},\cW_{\alpha_\bx})$ where $\bB_\bx$ is deduced from $\bB$ by a suitable rotation. At this point we use the continuity result of \cite[Theorem 4.5]{Pop13} for $(\bB,\alpha)\mapsto\En(\bB,\alpha)$ with respect to $\alpha\in(0,2\pi)$ and $\bB\in\dS^2$, which yields
\begin{equation}
\label{eq:wedgeinf}
   \min_{\beta_0\le\alpha\le2\pi-\beta_0,\ \|\bB\|=1} \En(\bB,\cW_{\alpha}) =: c(\beta_0)>0,
\end{equation}
where the diamagnetic inequality has been used to get the positivity. We deduce by homogeneity $\En(\bB,\Pi_{\bx})\ge c(\beta_0)\|\bB\|$.
\end{enumerate}
Finally 
\[
   \inf_{\bx\in\overline\Omega\setminus\gV} \En(\bB,\Pi_{\bx}) \ge \min\{c(\beta_0),\tfrac12\} \|\bB\|.
\]
Combined with the previous upper bound \eqref{eq:Env} at the vertex $\bv(\varepsilon)$, this estimate yields that condition \eqref{eq.condbis} is satisfied for $\a$ small enough, hence point \emph{a)} of Proposition \ref{prop:cornercon}.

Proof of point \emph{b)} of Proposition \ref{prop:cornercon}. Let us define
\[
   \Omega(\varepsilon) = \conea\cap\{\rx_3<1\}.
\]
By construction, we only have to check \eqref{eq:edgebound}. The edges of $\Omega(\varepsilon)$ can be classified in two sets:
\begin{enumerate}
\item The edges contained in those of $\conea$. We have proved in Section \ref{sec.esssp} that their opening converge to the opening angles of $\omega$ as $\varepsilon\to0$. 
\item The edges contained in the plane $\{\rx_3=1\}$. Their openings tend to $\frac\pi2$ as $\varepsilon\to0$.
\end{enumerate}
Hence \eqref{eq:edgebound}.

\section{Analogies with the Robin Laplacian}
\label{sec:rob}
We describe here some similarities of the Neumann magnetic Laplacian with the Robin Laplacian on corner domains. For a real parameter $\gamma$, this last operator acts as the Laplacian on functions satisfying the mixed boundary condition $\partial_{n}u-\gamma u=0$ where $\partial_{n}$ is the outward normal and $\gamma$ is a real parameter. The associated quadratic form is
$$u\mapsto \int_{\Omega}|\nabla u(x)|^2\rd x-\gamma\int_{\partial\Omega}|u(s)|^2 \rd s, \quad u\in\sH^{1}(\Omega).$$
Since the study initiated in \cite{LaOckSa98}, many works have been done in order to understand the asymptotics of the eigenpairs of this operator in the limit $\gamma\to+\infty$. It occurs that in this regime, the first eigenvalue $\lambda^{\Rob}_{\gamma}(\Omega)$ of this Robin Laplacian shares numerous common features with those of the magnetic Laplacian in the semi-classical limit. Levitin and Parnovski prove that for a corner domain $\Omega$ satisfying a uniform interior cone condition, there holds (see \cite[Theorem 3.2]{LevPar08})
\begin{equation}
\label{eq:Roblim}
    \lambda_{\gamma}^{\Rob}(\Omega)\underset{\gamma\to+\infty}{\sim}
    \gamma^2 \inf_{\bx \in \partial{\Omega}} \En^{\Rob}(\Pi_{\bx}),
\end{equation}
where, as before, $\En^{\Rob}(\Pi_{\bx})$ is the ground state energy of the model operator ($\gamma=1$) on the tangent cone $\Pi_{\bx}$ at $\bx$.
In fact, $\En^{\Rob}(\Pi_{\bx})<0$ for any boundary point $\bx$. This result leads to the same problematics as ours: compare the ground state energies of model operators on various tangent cones. When $\Pi_{\bx}$ is either a half-space or a wedge, $\En^{\Rob}(\Pi_{\bx})$ is explicit:
\begin{equation}
\label{eq:Robexp}
 \En^{\Rob}(\R^{3}_{+})=-1 \quad \mbox{and} \quad 
 \En^{\Rob}(\cW_{\alpha})=\left\{ 
 \begin{aligned}
 &-\sin^{-2}(\tfrac{\alpha}{2}) \ \ \mbox{if} \ \ \alpha\in (0,\pi]
 \\
  &-1\ \ \mbox{if} \ \ \alpha\in [\pi,2\pi).
 \end{aligned}
 \right.
\end{equation}
This shows, in some sense, that the Robin Laplacian is simpler for these cones. We notice that $\En^{\Rob}(\cW_{\alpha})\to-\infty$ as $\alpha\to0$. This fact should be compared to \eqref{eq:wedge}. The general idea behind this is an analogy between the degeneracy of the ground state energies, as follows: Whereas the ground energy (always positive) is going to $0$ for the magnetic Laplacian on sharp cones, the ground energy (always finite) of the Robin Laplacian goes to $-\infty$, as we shall explain below.
 
However, for cones of higher dimensions, no explicit expression like \eqref{eq:Robexp} is known for $\En^{\Rob}(\Pi_{\bx})$. In \cite[Section 5]{LevPar08}, a two-sided estimate is given for convex cones of dimension $\geq3$. The idea for this estimate is quite similar to our strategy: Given a suitable reference axis $\{\rx_{3}>0\}$ intersecting $\Pi\cap \dS^{2}$ at a point denoted by $\theta$, one defines the plane $P$ tangent to $\dS^{2}$ at $\theta$, so that the intersection $P\cap \Pi$ defines a section $\omega$ for which the cone $\Pi$ coincides with $\cone$ given by \eqref{def.cone}. Using polar coordinates $(\rho,\phi)\in \R^{+}\times \dS^{1}$ in the plane $P$ centered at $\theta$, one parametrizes the boundary of $\omega$ by a function $b$ through the relation $\rho=b(\phi)$. Then\footnote{In \cite[Theorem 5.1]{LevPar08}, the quantity $-\En^{\Rob}(\Pi)$ is estimated, so that the upper bound presented here, corresponds to the lower bound of the paper {\em loc. cit.}}, \cite[Theorem 5.1]{LevPar08} provides the upper bound
\begin{equation}
\label{eq:LP}
   \En^{\Rob}(\Pi)  \leq - \Bigg(\frac{\int_{\dS^{1}}\sigma(\phi) \,b(\phi)^2 \rd \phi}
   {\int_{\dS^{1}}b(\phi)^2 \rd \phi} \bigg)^2 \quad \mbox{with} \quad 
   \sigma(\phi)=\sqrt{1+b(\phi)^{-2}+b'(\phi)^2b(\phi)^{-4}} .
\end{equation}
Note that this estimate depends on the choice of the reference coordinate $\rx_{3}$, exactly as in our case, see Remark \ref{rem:invar}, and can be optimized by taking the infimum on $\theta$. 

Estimate \eqref{eq:LP} shows in particular that for our sharp cones $\conea$, the energy $\En^{\Rob}(\conea)$ goes to $-\infty$ like $\varepsilon^{-2}$ as $\varepsilon\to0$. This property is the analog of our upper bounds \eqref{eq:Ee}-\eqref{eq:homoa}. We expect that an analog of our formula \eqref{eq.Eness2} is valid, implying that there exists a finite limit for the bottom of the essential spectrum of the model Robin Laplacians defined on $\conea$, as $\varepsilon\to0$. This would provide similar conclusions for Robin problem and for the magnetic Laplacian.

\appendix
\section{Tangent cones and corner domains}
\label{SS:tangent}
Following \cite[Section 2]{Dau88} (see also \cite[Section 1]{BoDaPo14}), we recall the definition of corner domains.
We call a \emph{cone} any open subset $\Pi$ of $\R^n$ satisfying 
\[
   \forall\rho>0 \ \ \mbox{and}\ \ \bx\in\Pi,\quad \rho \bx\in\Pi,
\]
and the \emph{section} of the cone $\Pi$ is its subset $\Pi\cap\dS^{n-1}$. Note that $\dS^0=\{-1,1\}$.

\begin{definition}[\sc Tangent cone]
\label{D:Tangentcones}
Let $\Omega$ be an open subset of $M=\R^n$ or $\dS^n$. Let $\bx_{0}\in \overline{\Omega}$. 
The cone $\Pi_{\bx_{0}}$ is said to be \emph{tangent to} $\Omega$ at $\bx_{0}$ if there exists a local $\sC^\infty$  diffeomorphism $\diffeo^{\bx_{0}}$ which maps a neighborhood $\cU_{\bx_{0}}$ of $\bx_{0}$ in $M$ onto a neighborhood $\cV_{\bx_{0}}$ of $\bfz$ in $\R^n$ and such that
\begin{equation*}%\label{eq:diffeo}
   \diffeo^{\bx_{0}}(\bx_0) = \bfz,\quad
   \diffeo^{\bx_{0}}(\cU_{\bx_{0}}\cap\Omega) = \cV_{\bx_{0}}\cap\Pi_{\bx_{0}} 
   \quad\mbox{and}\quad
   \diffeo^{\bx_{0}}(\cU_{\bx_{0}}\cap\partial\Omega) = \cV_{\bx_{0}}\cap\partial\Pi_{\bx_{0}} .
\end{equation*}
\end{definition}

\begin{definition}[\sc Class of corner domains]
\label{D:cornerdomains}
For $M=\R^n$ or $\dS^n$, the classes of corner domains $\gD(M)$ and tangent cones $\gP_{n}$ are defined as follow: 

{\sc Initialization}: $\gP_0$ has one element, $\{0\}$.
$\gD(\dS^0)$ is formed by all subsets of $\dS^0$.

{\sc Recurrence}: For $n\ge1$,
\begin{enumerate}
\item $\Pi\in\gP_n$ if and only if the section of $\Pi$ belongs to $\gD(\dS^{n-1})$,
\item $\Omega\in\gD(M)$ if and only if for any $\bx_{0}\in\overline\Omega$, 
there exists a tangent cone $\Pi_{\bx_{0}}\in\gP_n$ to $\Omega$ at $\bx_{0}$.
\end{enumerate}
\end{definition}

\subsection*{Acknowledgments.} 
This work was partially supported by the ANR (Agence Nationale de la Recherche), project {\sc Nosevol} ANR-11-BS01-0019 and by the Centre Henri Lebesgue (program ``Investissements d'avenir'' -- n$^{\rm o}$ ANR-11-LABX-0020-01).

%\bibliographystyle{mnachrn}
%\bibliography{biblio}

\end{document}